\documentclass[a4paper,12pt,reqno]{amsart}
\usepackage{amsmath,amsfonts,amsthm,amssymb,color}
\usepackage[T1]{fontenc}
\usepackage{enumerate}

\usepackage{pdfsync}

\newcommand{\der}{\delta}

\newcommand{\iot}{\int_{0}^{t}}

\newcommand{\ott}{[0,T]}

\newcommand{\1}{{\bf 1}}

\newcommand{\R}{\mathbb R}
\newcommand{\N}{\mathbb N}

\newcommand{\be}{\mathbf{E}}

\newcommand{\bp}{\mathbf{P}}

\newcommand{\cac}{\mathcal C}

\newcommand{\cf}{\mathcal F}

\newcommand{\al}{\alpha}

\newcommand{\la}{\lambda}

\newcommand{\oom}{\Omega}

\newcommand{\si}{\sigma}
\newcommand{\te}{\vartheta}
\newcommand{\tte}{\Theta}

\newcommand{\lp}{\left(}
\newcommand{\rp}{\right)}

\newcommand{\lln}{\left|}
\newcommand{\rrn}{\right|}
\newcommand{\lla}{\left\langle}
\newcommand{\rra}{\right\rangle}

\newtheorem{theorem}{Theorem}[section]

\newtheorem{hypothesis}[theorem]{Hypothesis}
\newtheorem{lemma}[theorem]{Lemma}
\newtheorem{notation}[theorem]{Notation}

\newtheorem{proposition}[theorem]{Proposition}

\theoremstyle{remark}
\newtheorem{remark}[theorem]{Remark}
\theoremstyle{remark}

\newcommand{\bean}{\begin{eqnarray*}}
\newcommand{\eean}{\end{eqnarray*}}
\newcommand{\ben}{\begin{enumerate}}
\newcommand{\een}{\end{enumerate}}
\newcommand{\beq}{\begin{equation}}
\newcommand{\eeq}{\end{equation}}

\address{Andreas Neuenkirch, Fachbereich Mathematik,
TU Kaiserslautern, 
Postfach 3049,  
D-67663 Kaiserslautern, Germany {\tt neuenkirch@mathematik.uni-kl.de}                  
}

\address{Samy Tindel, Institut {\'E}lie Cartan Nancy, Universit\'e de Nancy 1, B.P. 239,
54506 Vand{\oe}uvre-l{\`e}s-Nancy Cedex, France {\tt tindel@iecn.u-nancy.fr}}

\begin{document}

\title[LSE for SDEs with additive fractional noise]{A least square-type procedure for parameter estimation in stochastic differential equations  \\ with additive fractional noise}

\date{\today}

\author{Andreas Neuenkirch \and Samy Tindel}
\date{\today}
\begin{abstract}
We study a least square-type estimator for  an unknown parameter in the drift coefficient of a 
stochastic differential equation with additive fractional noise of Hurst parameter $H>1/2$. The estimator is based on discrete time observations of the  stochastic differential equation, 
and using tools from ergodic theory  and stochastic analysis we derive its  strong consistency.
\end{abstract}

\thanks{S. Tindel is member of the BIGS (Biology, Genetics and Statistics) team at INRIA}

\date{\today}

\maketitle

\section{Introduction}
In this article, we will consider the following  $\R^d$-valued stochastic differential equation (SDE)
\begin{equation}\label{eq:sde}
Y_t=y_0+\iot  b(Y_s;\te_0) \, ds +  \sum_{j=1}^m \si_j B_t^{j} , \qquad t\in\ott.
\end{equation}
Here $y_0 \in \mathbb{R}^d$ is a given initial condition, $B=(B^{1}, \ldots, B^{m})$ is an $m$-dimensional fractional Brownian motion (fBm) with Hurst parameter $H \in (0,1)$, the unknown parameter 
$\te_0$ lies in a certain set $\tte$ which will be specified later on, $\{b(\cdot;\te), \, \te\in\tte \}$ is a family of drift coefficients with $b(\cdot;\te):\R^d\to\R $, and $\si_1, \ldots, \si_m \in \mathbb{R}^d$ are assumed to be  known diffusion coefficients. 

\smallskip

Let us recall  that $B$ is a centred Gaussian process defined on a complete probability space $(\oom,\cf,\bp)$. Its law is thus characterized by its covariance function, which is defined by
$$
\be \big(  B_t^i \, B_s^j \big) = \frac 12 \lp t^{2H} + s^{2H} - |t-s|^{2H}  \rp \, \1_{(i=j)},
\qquad s,t\in\R.
$$
The variance of the   increments of $B$ is  then given by
$$
\be \, | B_t^i - B_s^i |^2  = |t-s|^{2H}, \qquad s,t\in\R, \quad i=1,\ldots, m,
$$
 and this implies that  almost surely the fBm paths are $\gamma$-H\"{o}lder
  continuous for any $\gamma<H$. Furthermore, for $H=1/2$, fBm coincides with the usual Brownian motion, converting the family $\{B^H,\, H\in(0,1)\}$ into the most natural generalization of this classical process. In the current paper we assume that the Hurst coefficient satisfies $H>1/2$ and we focus on the estimation of the parameter $\te_0\in\tte$.  Note that the Hurst parameter and the diffusion coefficients can be estimated via the quadratic variation of $Y$, see e.g. \cite{begyn,coeur,IL}.

\smallskip

Estimators for the unknown parameter in equation (\ref{eq:sde}) based on continuous observation of $Y$ have been studied e.g. in \cite{BEO,HN_foup,KL,LB,PL,Rao,TV}.  Estimators based on discrete  time data, which are important for practical applications, are then obtained via discretization.
However, to the best of our knowledge no genuine estimators based on discrete time data have been analyzed yet. 

\smallskip

We propose here a least square estimator for $\te_0$ based on discrete observations of the process $Y$ at times $\{t_k;\, 0\le k \le n\}$. For simplicity, we shall take equally spaced observation times with $t_{k+1}-t_{k}=\kappa \, n^{-\alpha}:=\al_n$ with  given  $\alpha \in (0,1), \kappa > 0$.  We call our method least square-type procedure, insofar as we consider a quadratic statistics of the form
\begin{equation}\label{eq:def-Qn}
Q_n(\te)= \frac{1}{n \alpha_n^2}\sum_{k=0}^{n-1} \left(
\lln  \der Y_{t_{k} t_{k+1}} -  b(Y_{t_k};\te) \al_n\rrn^2 - \|\sigma \|^2 \alpha_n^{2H} \right),
\end{equation}
where $\der Y_{u_1u_2}:=Y_{u_2}-Y_{u_1}$ for any $0\le u_1\le u_2\le T$
and $\|\sigma \|^2= \sum_{j=1}^m |\sigma_j|^2$.

\smallskip

Let us now describe the assumptions under which we shall work, starting from a standard hypothesis on the parameter set $\tte$:

\begin{hypothesis}\label{hyp:Theta}
The set $\tte$ is compactly embedded in $\R^q$ for a given $q\ge 1$.
\end{hypothesis}

In order to describe the assumptions on our coefficients $b$, we will use the following notation for partial derivatives:

\begin{notation}
Let $f:\R^d\times\tte\to\R$ be a $\cac^{p_1,p_2}$ function for $p_1,p_2\ge 1$. Then for any tuple $(i_1,\ldots i_{p_1})\in\{1,\ldots,d\}^{p_1}$, we set $\partial_{i_1\ldots i_{p_1}} f$ for $\frac{\partial^{p_1} f}{\partial x_{i_1}\ldots \partial x_{i_{p_1}}}$. Moreover, we will write 
$\partial_x  f$ resp. $\partial_{\vartheta} f$ for the Jacobi-matrices
$(\partial_{x_1} f, \ldots, \partial_{x_d}f)$ and 
$(\partial_{\te_1} f, \ldots, \partial_{\te_q}f) $ .
\end{notation}

With this notation in mind, our drift coefficients and their derivatives will satisfy a polynomial growth condition, plus an inward condition which is traditional for estimation procedures in the Brownian diffusion case (see e.g \cite{flor,Kas}):
\begin{hypothesis}\label{hyp:f-coercive} We have $b \in \mathcal{C}^{1,1} (\mathbb{R}^d \times \tte; \mathbb{R}^d)$ and 
there exist constants $c_1, c_2>0$ and $N \in \mathbb{N}$ such that:

\noindent
\emph{(i)} For every $x,y\in\R^{d}$ and $\te\in\tte$ we have
\begin{equation*}
\lla  b(x;\te)-b(y;\te), \, x-y\rra \le  - c_1 |x-y|^2
\end{equation*}

\noindent
\emph{(ii)} For every $x\in\R^{d}$ and $\te\in\tte$ the following growth bounds are satisfied:
\begin{equation*} 
|b(x;\te)| \le c_2 \lp  1+|x|^N\rp,
\quad
|\partial_x b(x;\te)| \le c_2 \lp  1+|x|^N\rp,  \quad   |\partial_{\vartheta} b(x;\te)| \le c_2 \lp  1+|x|^N \rp.
\end{equation*}
\end{hypothesis}

As a consequence of the above assumptions on the drift coefficient and the initial condition,  for given $\vartheta_0 \in \tte$  the solution of equation (\ref{eq:sde})  converges  for $t \rightarrow \infty$ to a stationary and ergodic stochastic process $(\overline{Y}_t, t \geq 0)$, 
see the next section. 

\smallskip

Finally, we also assume that our-drift coefficient is of gradient-type, i.e.:

 \begin{hypothesis}\label{hyp:f-gradtype} 
There exists a function $U \in \mathcal{C}^{2,1}(\mathbb{R}^d \times \tte; \mathbb{R})$ such that
$$ \partial_x U (x; \te)= b(x; \te), \qquad x \in  \mathbb{R}^d , \, \, \te \in \tte.$$
\end{hypothesis}

With those assumptions in mind, we obtain the following convergence result:

\begin{theorem}\label{thm:lse-convergence_prev}
Assume that the Hypotheses \ref{hyp:Theta}, \ref{hyp:f-coercive} and \ref{hyp:f-gradtype}  are satisfied for equation (\ref{eq:sde}) and that we moreover have $H>1/2$. Let $Q_n(\te)$ be defined by (\ref{eq:def-Qn}). Then we have 
\begin{equation}\label{eq:lim-fbm-Q}
\sup_{ \vartheta \in \Theta}  \left| Q_n(\vartheta) -    \left( \mathbf{E} \, |  b( \overline{Y}_0 ;\te_0)|^2 - \mathbf{E}|  b( \overline{Y}_0 ;\te) |^{2}    \right)   \right| \rightarrow 0
\end{equation}
in the $\bp$-almost sure sense.
\end{theorem}

This convergence is in contrast to the case $H=1/2$, i.e. to the case of SDEs with additive Brownian noise. There it holds 
\begin{equation}\label{eq:lim-brownian-Q}
\sup_{ \vartheta \in \Theta}  \left| Q_n(\vartheta) -    \left( \mathbf{E} \, |  b( \overline{Y}_0 ;\te_0)-  b( \overline{Y}_0 ;\te) |^{2}    \right)   \right| \rightarrow 0 
\end{equation}
in the $\bp$-almost sure sense, see Remark \ref{H12},
and usually the consistent least squares estimator
$$ \textrm{argmin}_{\te \in \tte} \sum_{k=0}^{n-1} 
\lln  \der Y_{t_{k} t_{k+1}} -  b(Y_{t_k};\te) \al_n\rrn^2,$$
is considered, see e.g. \cite{flor,Kas}. The difference in the limits \eqref{eq:lim-fbm-Q} and \eqref{eq:lim-brownian-Q} is due to the higher smoothness and long-range dependence of fractional Brownian motion for $H>1/2$. Our estimator can thus be seen as a ''zero squares'' estimator instead of a classical least square estimator. In order to show its convergence, we shall work under the following natural identifiability  assumption:
\begin{hypothesis}\label{hyp:identifiability}
For any $\te_0\in\tte$, we have 
\begin{equation*}
\mathbf{E} \, |  b( \overline{Y}_0 ;\te_0)|^2  = \mathbf{E}|  b( \overline{Y}_0 ;\te) |^{2} 
\quad\mbox{iff}\quad
\te=\te_0.
\end{equation*}
\end{hypothesis}

 With this additional Hypothesis, the main result of the current article is the consistency of the zero squares estimator based on the statistics $Q_n$:

\begin{theorem}\label{thm:lse-convergence}
Assume that the Hypotheses \ref{hyp:Theta}, \ref{hyp:f-coercive}, \ref{hyp:f-gradtype} and \ref{hyp:identifiability} are satisfied for equation (\ref{eq:sde}) and let $H>1/2$. Let $Q_n(\te)$ be defined by (\ref{eq:def-Qn}), and let $\widehat{\te}_n={\rm argmin}_{\te\in\tte} \left| Q_n(\te) \right|$. Then for any $\te_0\in\tte$, we have $\lim_{n\to\infty}\widehat{\te}_n=\te_0$ in the $\bp$-almost sure sense.
\end{theorem}

\smallskip

Note that minimizing $|Q_n(\vartheta)|$ is of course equivalent to finding the zero of $Q_n(\vartheta)$.
Let us shortly compare Theorem  \ref{thm:lse-convergence}  with the existing literature on estimation procedures for fBm driven equations:

\smallskip

\noindent
(i) Most of the previous results, see e.g. \cite{BEO,HN_foup,KL,Rao}, deal with the one-dimensional fractional Ornstein-Uhlenbeck process in a continuous observation setting. In particular, for this process
simple  continuous time least-square estimators are obtained in \cite{BEO, HN_foup}, for which also covergence rates and  asymptotic error distributions are derived.
Compared to these results our estimation procedure covers a broad class of ergodic multi-dimensional equations and relies on discrete data only.

\smallskip

\noindent
(ii) A general estimation procedure based on moment matching is established in \cite{PL}. However, the main assumption in \cite{PL} is that many independent observations of sample paths over a short time interval are available, which is not the case in  many practical situations where rather one sample path is discretely observed for a long time period. Let us also mention the article~\cite{CT}, in which a general discrete data maximum likelihood type procedure has been designed for  parameter estimation in both the drift and diffusion coefficients, however without proof of consistency. 

\smallskip

\noindent
(iii) Our current work probably  compares  best with the maximum likelihood estimator analyzed in \cite{TV}. The latter pioneering reference focused on one-dimensional SDEs of the form
$$ dY_t = \vartheta_0 h(Y_t) \, dt +  dB_t$$
with $h:\R \rightarrow \R$ satisfying suitable regularity assumptions. Strong consistency is obtained for the continuous time estimator and also for a discretized version of the estimator.
However, the discretized estimator involves rather complicated operators related to the kernel functions arising in  the Wiener-integral representation of fBm,  which are avoided in  our approach.  Moreover, in contrast to \cite{TV} the consistency proof for our estimator does not rely on  Malliavin calculus methods.

\smallskip

So, in view of the existing results in the literature, Theorem \ref{thm:lse-convergence} can  be seen a step towards simple and implementable parameter estimation procedures for SDEs driven by fBm.

\smallskip

Finally, let us comment on the assumptions we have imposed on the drift coefficient and on the Hurst parameter:

\smallskip

\noindent
(a) The hypothesis of Theorem \ref{thm:lse-convergence_prev} are standard for the case $H=1/2$, except Hypothesis~\ref{hyp:f-gradtype} which restricts us to gradient-type drift coefficients. We require this condition to show an ergodic-type result for weighted sums of the increments of fBm, see Lemma~\ref{ergodic_sum:fbm}. However, this  Hypothesis \ref{hyp:f-gradtype} is also implicitly  present in the additional condition of \cite[Theorem 1]{Kas}.

\smallskip

\noindent
(b) It can  easily  be shown that whenever $\te$ is a one-dimensional coefficient (namely for $q=1$), Hypothesis~\ref{hyp:identifiability} is satisfied if the drift coefficient is of the form $b(x;\te)=\te h(x)$ for some $h: \R^d \rightarrow \R^d$ and if the stationary solution is non-degenerate, i.e. we have $\mathbf{E}|\overline{Y}_0|^2 \neq 0$. The latter conditions hold in particular in the case of the ergodic fractional Ornstein-Uhlenbeck process.
It would be nice to obtain criteria for richer classes of examples, but this would rely on differentiability and non-degeneracy properties of the map $\te \mapsto \mathbf{E}|  b( \overline{Y}_0 ;\te) |^{2}$ (see \cite{HM} in the Markovian case). We wish to investigate this question in future works.

\smallskip

\noindent
(c) Even if the noise enters additively in our equation, we still need the assumption $H>1/2$ in order to prove Theorem \ref{thm:lse-convergence}. Indeed, this hypothesis ensures the convergence of some deterministic and stochastic Riemann sums in the computations below (see Remark~\ref{rmk:H-greater-one-half} for further details). Whether an adaptation of the proposed zero squares estimator is also convergent in the case $H<1/2$ remains an open problem.

\medskip

Let us finish this introduction with the simplest example of an equation which satisfies the above assumptions: namely  the one-dimensional fractional Ornstein-Uhlenbeck process given by 
$$ dY_t = \vartheta_0 Y_t \,dt + dB_t, \qquad Y_0=y_0 \in \mathbb{R} $$
with $\vartheta_0 < 0$. The solution of this SDE reads as
$$ Y_t=  y_0 \exp( \vartheta_0 t) + \exp( \vartheta_0 t ) \int_0^t \exp(-\vartheta_0 s) \, dB_s .$$
 For $t \rightarrow \infty$ this process converges to the stationary fractional Ornstein-Uhlenbeck process
$$ \exp( \vartheta_0 t) \int_{- \infty}^t \exp( - \vartheta_0 s) \, dB_s, \qquad  t \geq 0, $$ see e.g. \cite{GKN}.
Here straightforward computations yield the explicit  estimator
\begin{align*}  
\widehat{\te}_n =  \frac{ \sum_{k=0}^{n-1}Y_{t_k}  \delta Y_{t_k t_{k+1}} } { \sum_{k=0}^{n-1}Y_{t_k}^2  \alpha_n} - \sqrt{ \left( \frac{ \sum_{k=0}^{n-1}Y_{t_k}  \delta Y_{t_k t_{k+1}} } { \sum_{k=0}^{n-1}Y_{t_k}^2  \alpha_n} \right)^2- \frac{ \sum_{k=0}^{n-1}\left( |\delta Y_{t_k t_{k+1}}|^2 - \alpha_n^{2H} \right) } { \sum_{k=0}^{n-1}Y_{t_k}^2  \alpha_n^2 } } .
\end{align*}
 Notice that even in the case of the fractional Ornstein-Uhlenbeck process, we could neither proof the consistency nor show
the inconsistency of our estimator for $H<1/2$.

\smallskip

\smallskip

The remainder of this paper is structured as follows: In Section \ref{sec:preliminaries} we give some auxiliary results on stochastic calculus for fractional Brownian motion. Section \ref{sec:proof-thms} is then devoted to the proof of our main theorems.

\section{Auxiliary Results}\label{sec:preliminaries}

\subsection{Ergodic Properties of the SDE}
To deduce the ergodic properties of SDE (\ref{eq:sde}) we will work
without loss of generality  on the canonical probability space 
  $(\Omega, \mathcal{F}, \mathbf{P})$, i.e. $\Omega=C_{0}(\mathbb{R}, \mathbb{R}^{m})$ equipped with the compact open topology, $\mathcal{F}$ is
 the corresponding Borel-$\sigma$-algebra and $\mathbf{P}$ is the distribution of the fractional Brownian motion $B$, which is consequently given here by the canonical process $B_t(\omega)=\omega(t)$, $t \in \mathbb{R}$.
Together  with the shift operators $\theta_t: \Omega \rightarrow \Omega$ defined by
$$ \theta_t\omega(\cdot)= \omega(\cdot +t)- \omega(t), \qquad t \in \mathbb{R}, \quad \omega \in \Omega,$$
the canonical probability space defines an  ergodic metric dynamical system,
see e.g. \cite{GA_S}. In
particular, the measure $\mathbf{P}$ is invariant to the shift
operators $\theta_t$, i.e. the shifted process $(B_{s}(\theta_t
\cdot))_{s \in \mathbb{R}}$ is still an $m$-dimensional fractional
Brownian motion and for any integrable random variable $F: \Omega \rightarrow \mathbb{R}$ we have
$$ \lim_{T \rightarrow \infty} \frac{1}{T} \int_0^T F(\theta_t(\omega)) \, dt  = \mathbf{E} \, F $$
for $\mathbf{P}$-almost all $\omega \in \Omega$. 
Owing to the results in Section 4 of \cite{GKN} we have the following:

\begin{theorem}\label{thm:attractor}
Let  Hypothesis \ref{hyp:f-coercive} hold. Then  for any $\te_0\in\tte$ the following holds:

\smallskip

\noindent
\emph{(i)} Equation \eqref{eq:sde} admits a unique solution $Y$ in $C^{\la}(\R_+;\R^d)$ for all $\lambda <H$.

\smallskip

\noindent
\emph{(ii)} 
There exists a random variable  $\overline{Y}: \Omega \rightarrow \mathbb{R}^d $  such that
$$   \lim_{t \rightarrow \infty}  \,\, | Y_t(\omega) - \overline{Y}(\theta_t \omega) | = 0 $$
for $\mathbf{P}$-almost all $\omega \in \Omega$. Moreover, we have
$ \mathbf{E} |\overline{Y}|^p < \infty $
for all $p \geq 1$.
\end{theorem}

Note that the law of $\overline{Y}$ must coincide with the  attracting invariant  measure for (\ref{eq:sde}) given in \cite{Ha}, see also \cite{HO,HP}.
Moreover, proceeding as  in \cite{GKN}  we  have:

\begin{proposition}\label{prop:bnd-moments-Y}
Assume Hypothesis \ref{hyp:f-coercive} holds true. Then for any $\te_0\in\tte$ and $p\ge 1$  there exist constants $c_p, k_p >0$ such that 
\begin{equation*}
\be \, |Y_t|^{p}  \le c_p, \qquad \be \, |Y_t-Y_s|^{p}  \le k_p |t-s|^{pH}, 
 \qquad\mbox{for all}\quad s,t\ge 0.
\end{equation*}
\end{proposition}

The integrability of $\overline{Y}$ now
 implies the  ergodicity of equation (\ref{eq:sde}):

\begin{proposition}\label{prop:ergod-Y}
Assume Hypothesis \ref{hyp:f-coercive} holds true. Then for any  $\te_0\in\tte$ and any $f \in \mathcal{C}^{1}(\mathbb{R}^d; \mathbb{R})$ such that
\begin{equation*} 
|f(x)| + |\partial_x f(x)| \le c \lp  1+|x|^N\rp, \qquad x \in \R^d,
\end{equation*}
for some $c>0$, $N \in \mathbb{N}$,  we have
\begin{equation*}
  \lim_{T \rightarrow \infty} \frac{1}{T} \int_0^T f(Y_t) \, dt  = \mathbf{E}f(\overline{Y}) \qquad  \mathbf{P}\textrm{-}a.s.
\end{equation*}
\end{proposition}

\begin{proof}
Since the shift operator is ergodic and $f$ has polynomial growth, we have
$$ \lim_{T \rightarrow \infty} \frac{1}{T} \int_0^T f(\overline{Y}(\theta_t)) \, dt  = \mathbf{E}f(\overline{Y}) \qquad  \mathbf{P}\textrm{-}a.s.$$ Moreover, since
$$   \lim_{t \rightarrow \infty}  \, | Y_t(\omega) - \overline{Y}(\theta_t \omega) | = 0 $$ by Theorem \ref{thm:attractor} and $f$ is polynomially Lipschitz, the assertion easily follows.

\end{proof}

\smallskip

\subsection{Generalized Riemann-Stieltjes Integrals}\label{RS_integrals}

We set
$$  \| f \|_{\infty;[a,b]}= \sup_{t \in [a,b]} |f(t)|, \qquad |f |_{\lambda;[a,b]} = \sup_{s,t \in [a,b]} \frac{|f(t)-f(s)|}{|t-s|^{\lambda}}$$
where $f: \R \rightarrow \R^n$ and $\lambda \in (0,1)$.

\smallskip

Now, let $f \in C^{\lambda}([a,b]; \R)$ and
$g \in C^{\mu}([a,b]; \R)$ with $\lambda + \mu >1$.  Then it is well known that the Riemann-Stieltjes integral $\int_{a}^{b} f(x) \, dg(x)$  exists, see e.g.  \cite{young}. Also, the classical chain rule for the change of variables remains valid, see e.g. \cite{zaehle}:
 Let $f \in C^{\lambda}([a,b];\R)$ with
 $\lambda >1/2$ and $F \in C^{1}(\R;\R)$.  Then we have
\begin{align}{\label{ito_proto}} 
 & F(f(y))  -F(f(a)) =   \int_{a}^{y} F'(f(x)) \, df(x), \qquad y \in [a,b].
\end{align}   Moreover, one   has a density type formula: let $f,h \in C^{\lambda}([a,b]; \R)$ and $g \in C^{\mu}([a,b]; \R)$  with $\lambda + \mu >1$.  Then for
$$ \varphi:[a,b] \rightarrow \R, \quad \varphi(y)=\int_{a}^{y} f(x) \, dg(x), \qquad  y \in [a,b],$$  we have
\begin{align} \label{dens_form} \int_{a}^{b} h(x) \, d \varphi(x) = \int_{a}^{b} h(x) f(x) \, d g(x). \end{align}

For later use, we also note the following estimate, which can be found e.g. in \cite{young}.
\begin{proposition}\label{young_est} Let $f,g$ as above. There exists a constant $c_{\lambda, \mu}$ (independent of $a,b$) such that
$$        \left|  \int_{a}^{b} (f(s)-f(a) ) d g(s) \right| \leq c_{\lambda, \mu }  | f|_{\lambda;[a,b]} |g|_{\mu;[a,b]} |b-a|^{\lambda+ \mu} $$  holds
for all $a,b \in [0,\infty)$.
\end{proposition}

\subsection{The Garcia-Rademich-Rumsey Lemma}
We will use the following variant of the Garcia-Rademich-Rumsey Lemma \cite{grr}:

\begin{lemma}\label{grr_lemma} Let $q>1$, $\alpha \in (1/q,1)$ and $f: [0, \infty) \rightarrow \mathbb{R}$ be a continuous function. Then there exists a constant $c_{\alpha, q}>0$, depending only on $\alpha, q$, such that
$$ | f |_{[s,t]; \alpha-1/q}^q \leq c_{\alpha,q} \int_{s}^t \int_s^t \frac{|f(u)-f(v)|^q}{|u-v|^{1+q \alpha}} \, du  \, dv. $$
\end{lemma}

\smallskip

\subsection{A Lemma on Pathwise Convergence Rates}
The following  Lemma (see e.g. \cite{lms}), which is a direct consequence of the Borel-Cantelli Lemma, allows us to turn convergence rates in the 
$p$-th mean  into  pathwise convergence rates.

\begin{lemma}\label{pathwise}
Let $\alpha >0$, $p_0 \in \mathbb{N}$ and $c_p \in [0, \infty)$ for $p \geq p_0$.  In addition, let
$Z_{n}$, $n \in \N$, be a sequence of  random variables such that $$(\mathbf{E}
|Z_{n}|^{p})^{1/p} \leq c_p \cdot n^{-\alpha}$$ 
for all $p \geq p_0$ and all $n \in \N$. Then for all
$\varepsilon > 0$ there  exists a  
random variable  $\eta_{\varepsilon}$ such that
$$ |Z_{n}| \leq \eta_{\varepsilon} \cdot n^{-\alpha + \varepsilon}
\qquad a.s. $$
for all $n \in \N$.  Moreover,   $\mathbf{E} |\eta_{\varepsilon}|^{p} < \infty$ for all
$p \geq 1$.
\end{lemma}

\smallskip

\subsection{Quadratic Variations of Fractional Brownian Motion}\label{sec:quadratic-var}

The following result for the behavior of the quadratic variations of a one-dimensional fractional Brownian motion  $\beta$ with Hurst parameter $H$ is well known, see e.g  \cite{VT}.
Indeed, for $H<3/4$ we have
\begin{equation}  \label{BM1}
  \mathbf{E} \Big| \frac{1}{n}\sum_{k=0}^{n-1}\big[|\delta_{kk+1}\beta|^2-1\big] \Big|^2
\leq c_H \cdot \frac{1}{n},
\end{equation}
while for  $H=\frac{3}{4}$, $n>1$, it holds
\begin{equation}  \label{BM2}  \mathbf{E} \Big| \frac{1}{n} \sum_{k=0}^{n-1}\big[|\delta_{kk+1}\beta|^2-1\big] \Big|^2
\leq  c_{3/4}  \cdot \frac{ \log(n)}{n}.
\end{equation}
Finally, if $H\in(\frac{3}{4},1)$ then
\begin{equation}  \label{BM3}
 \mathbf{E} \Big| \frac{1}{n}\sum_{k=0}^{n-1}\big[|\delta_{kk+1}\beta|^2-1\big] \Big|^2
\leq  c_{H}  \cdot \frac{ 1}{n^{4-4H}}.
\end{equation}
Here, $c_H>0$ denotes a constant depending only
on $H$.

\bigskip

\section{Proof of Theorems  \ref{thm:lse-convergence_prev} and \ref{thm:lse-convergence} } 
\label{sec:proof-thms}
We will denote
constants, whose particular value is not important (and which do not depend on $\te$ or $n$) by $c$, regardless of their value.

\smallskip

Recall that
$$ Q_n(\te)= \frac{1}{n \alpha_n^2}\sum_{k=0}^{n-1} \left(
\lln  \der Y_{t_{k} t_{k+1}} -  b(Y_{t_k};\te) \al_n\rrn^2 - \|\sigma \|^2 \alpha_n^{2H} \right).$$
For $t\ge 0$, setting 
\begin{equation*}
F_t = \sum_{j=1}^m \sigma_j B_t^{(j)},
\quad\mbox{and}\quad
r_k = \int_{t_k}^{t_{k+1}} \left(b(Y_u;\te_0)- b(Y_{t_k};\te_0)\right) \, du
\end{equation*}
and moreover using the notation
 $$\delta_{\te_0 \te}b(x)=b(x;\te)-b(x;\te_0),
 \quad\mbox{and}\quad
 \delta F_{t_kt_{k+1}} = F_{t_{k+1}} - F_{t_{k}},
 $$
it is readily checked that
\begin{align}\label{eq:exp-Q-n1}
Q_n(\te) & =   \frac{1}{n \alpha_n^2}\sum_{k=0}^{n-1}  | \delta_{\te_0 \te}b(Y_{t_k})|^2 \al_n^2  -  \frac{2}{n \alpha_n^2}\sum_{k=0}^{n-1}  \langle  \delta_{\te_0 \te}b(Y_{t_k}),  \delta F_{t_kt_{k+1}}\rangle \al_n \notag
\\ & \qquad +  \frac{1}{n \alpha_n^2}\sum_{k=0}^{n-1} \left(
  |\der F_{t_{k} t_{k+1}}|^2  - \|\sigma \|^2 \alpha_n^{2H}  \right )   +  \frac{1}{n \alpha_n^2}\sum_{k=0}^{n-1} \left|r_k \right|^2 \\& \qquad -  \frac{2}{n \alpha_n^2}\sum_{k=0}^{n-1} \langle  \delta_{\te_0 \te}b(Y_{t_k}), r_k \rangle \alpha_n + \frac{2}{n \alpha_n^2}\sum_{k=0}^{n-1} \langle  \delta F_{t_kt_{k+1}}, r_k \rangle .\notag
\end{align}
Note that our assumptions on the drift coefficient imply that
$$  \sup_{\te \in \tte} |b(x;\vartheta)- b(y;\vartheta) | \leq c \big(1+ |x|^N+ |y|^N \big) \cdot |x-y|$$
for all $x,y \in \R^d$
and
$$  |b(x; \te_1) -b(x; \te_2) | \leq c \big (1+ |x|^N \big) \cdot |\te_1 -\te_2| $$
for all $x \in \R^d$ and $\te_1,\te_2 \in \tte$.
So, straightforward estimations using Proposition \ref{prop:bnd-moments-Y} give 
$$ \mathbf{E}|r_k|^p \leq c \cdot \alpha_n^{p(1+H)}.$$
Hence for all $p\ge 1$ it holds 
 $$ \mathbf{E}  \left| \sum_{k=0}^{n-1} \left|r_k \right|^2 \right|^p \leq c \cdot n^p \alpha_n^{2p(1+H)},$$
and  Lemma \ref{pathwise} implies 
\begin{equation}\label{eq:rk1}
\lim_{n \rightarrow \infty}   \frac{1}{n \alpha_n^2} \sum_{k=0}^{n-1} \left|r_k  \right|^2  =0 \qquad 
{\bf P}\textrm{-}a.s.
\end{equation}
Using Proposition \ref{prop:bnd-moments-Y} and Lemma \ref{pathwise} again, it follows similarly
\begin{equation}\label{eq:rk2}
 \lim_{n \rightarrow \infty}   \sup_{\te \in \tte} \frac{2}{n \alpha_n^2} \left| \sum_{k=0}^{n-1} \langle  \delta_{\te_0 \te}b(Y_{t_k}), r_k \rangle \alpha_n  \right|  =0 \qquad
{\bf P}\textrm{-}a.s.
 \end{equation}
 and, since $H>1/2$, we also have
\begin{equation}\label{eq:rk3}
 \lim_{n \rightarrow \infty}   \frac{2}{n \alpha_n^2} \left|  \sum_{k=0}^{n-1} \langle  \delta F_{t_kt_{k+1}}, r_k \rangle \right|  =0 \qquad 
 {\bf P}\textrm{-}a.s.
 \end{equation}
Plugging relations \eqref{eq:rk1}--\eqref{eq:rk3} into \eqref{eq:exp-Q-n1}, we have obtained that
\begin{equation}\label{eq:exp-Q-n2}
Q_n(\te)= Q_n^{(1)}(\te) - 2  Q_n^{(2)}(\te) + Q_n^{(3)} + R_n(\te),
\end{equation}
where $\lim_{n \rightarrow \infty} \sup_{\te \in \tte} |R_n(\te)| = 0$ in the $\bp$-almost sure sense and 
\begin{equation*}
Q_n^{(1)}(\te) =
\frac{1}{n}\sum_{k=0}^{n-1}  | \delta_{\te_0 \te}b(Y_{t_k})|^2,
\quad
Q_n^{(2)}(\te) =
\frac{1}{n \alpha_n}\sum_{k=0}^{n-1}  \langle  \delta_{\te_0 \te}b(Y_{t_k}),  \delta F_{t_kt_{k+1}}\rangle 
\end{equation*}
and
\begin{equation*}
Q_n^{(3)} =
\frac{1}{n \alpha_n^2}\sum_{k=0}^{n-1} \left(
  |\der F_{t_{k} t_{k+1}}|^2  - \|\sigma \|^2 \alpha_n^{2H}  \right ) .
\end{equation*}

\smallskip

The  treatment  of the  terms $Q_n^{(1)}(\te), Q_n^{(2)}(\te)$ and $Q_n^{(3)}$ will be carried out in the following series of Lemmata.
We first show a discrete version of  Proposition \ref{prop:ergod-Y}:

\begin{lemma}\label{ergodic:disc}  Let $f \in \mathcal{C}^{1,1}(\mathbb{R}^d \times \tte; \mathbb{R}^d)$ be  a function such that 
\begin{equation*} 
|f(x;\te)| \le c \lp  1+|x|^N\rp,
\quad
|\partial_x f(x;\te)| \le c \lp  1+|x|^N\rp,  \quad   |\partial_{\vartheta} f(x;\te)| \le c \lp  1+|x|^N \rp
\end{equation*}
for some $c>0, N \in \mathbb{N}$, independent of $\te \in \tte$. Then
we have 
$$ \sup_{\vartheta \in \tte} \left|  \frac{1}{n} \sum_{k=0}^{n-1}  |f(Y_{t_k}; \vartheta)|^2 - \mathbf{E}| f(\overline{Y}; \te)|^2  \right| \rightarrow 0  \qquad {\bf P}\textrm{-}a.s.$$
In particular, we have
\begin{equation*}
\sup_{\vartheta \in \tte} \left|  Q_n^{(1)}(\te) - \mathbf{E}| \delta_{\te_0\te}b(\overline{Y})|^2  \right| \rightarrow 0  \qquad {\bf P}\textrm{-}a.s.
\end{equation*}

\end{lemma}

\begin{proof} Let $T_n=n \alpha_n$
and set
$$  {V}_n(\vartheta)= \frac{1}{T_n} \int_0^{T_n} |f(Y_{s}; \te)|^2  \, ds. $$
The ergodicity of $Y$ yields that there exists a set $A_1 \in \mathcal{F}$ with full measure such that
$$  \lim_{n \rightarrow  \infty}  {V}_n(\vartheta)(\omega)    = \mathbf{E}|f(\overline{Y}; \te)|^2  $$
for all $\vartheta \in \Theta \cap \mathbb{Q}^q$ and all $\omega \in A_1$.
The assumptions on $f$ give
\begin{align} \label{equi_cont}  |{V}_n(\vartheta_1) - {V}_n(\vartheta_2)| \leq c \cdot  \left( 1 + \frac{1}{T_n} \int_0^{T_n} |Y_s|^{2N} \, ds \right) \cdot |\vartheta_1 - \vartheta_2| , \end{align}
 so ${V}_n$ is Lipschitz continuous in $\vartheta$ and thus
\begin{align*}
 \sup_{\vartheta \in \tte} \left|  {V}_n(\vartheta) -  \mathbf{E}|f(\overline{Y};\te)|^2        \right| = \sup_{\vartheta \in \tte \cap \mathbb{Q}^q} \left|  {V}_n(\vartheta) -  \mathbf{E}|f(\overline{Y};\te)|^2 \right|. \end{align*}
However, from (\ref{equi_cont}) and the ergodicity of $Y$, it also follows that there exists a set $A_2 \in \mathcal{F}$ with $\mathbf{P}(A_2)=1$ in which the family of random functions ${V}_n : \Theta \rightarrow \mathbb{R}$, $n \in \mathbb{N}$, is equicontinuous, and hence the Arzela-Ascoli Theorem yields the desired uniform convergence, i.e.
\begin{align}\label{uniform_ergod} \lim_{n \rightarrow \infty} \sup_{\vartheta \in \tte \cap \mathbb{Q}^q} \left|  {V}_n(\vartheta) -  \mathbf{E}|f(\overline{Y};\te)|^2 \right| =0  \qquad {\bf P}\textrm{-}a.s. \end{align}

Setting
$$ G_n(t;\vartheta)= |f(Y_t;\te)|^2 - |f(Y_{t_k};\te)|^2 , \qquad t \in [t_k,t_{k+1}), \qquad k=0,1, \ldots,
$$ it remains to show that
$$ \frac{1}{T_n} \int_{0}^{T_n} \sup_{ \vartheta \in \tte} |G_n(t;\vartheta)| \, dt \rightarrow 0 \qquad {\bf P}\textrm{-}a.s. $$ 
To this aim, the assumptions on $f$ imply that
$$   \sup_{\vartheta \in \tte} | G_n(t;\vartheta)| \leq c \cdot (1 + Y_t^{2N} + Y_{t_k}^{2N}) \cdot |Y_t - Y_{t_k}|.$$   Using  Proposition \ref{prop:bnd-moments-Y} and H\"older's inequality we obtain
\begin{align} \label{est_G_n} 
\sup_{t \geq 0} \, \mathbf{E} \sup_{\vartheta \in \tte} |G_n(t;\vartheta)|^p \leq c \cdot \alpha_n^{pH} \end{align}
for all $p \geq 1$. Now, Jensen's inequality gives
\begin{align*} \mathbf{E} \left| \frac{1}{T_n} \int_{0}^{T_n} \sup_{ \vartheta \in \tte} |G_n(t;\vartheta)| \, dt \right|^p  & \leq   \frac{1}{T_n} \int_0^{T_n} \mathbf{E}  \sup_{ \vartheta \in \tte} |G_n(t;\vartheta)|^p \, dt,
\end{align*}
so (\ref{est_G_n}) yields
\begin{align*} \mathbf{E} \left| \frac{1}{T_n} \int_{0}^{T_n} \sup_{ \vartheta \in \tte} |G_n(t;\vartheta)| \, dt \right|^p  & \leq  c \cdot \alpha_n^{pH}
\end{align*}
for all $p \geq 1$. Lemma \ref{pathwise} implies  
$$ \frac{1}{T_n} \int_{0}^{T_n} \sup_{ \vartheta \in \tte} |G_n(t;\vartheta)| \, dt \rightarrow 0 \qquad {\bf P}\textrm{-}a.s. $$
for $n \rightarrow \infty$.
\end{proof}

We have a similar ergodic result for weighted sums of the increments of the process $F$.

\begin{lemma}  \label{ergodic_sum:fbm}  Let $f \in \mathcal{C}^{1,1}(\mathbb{R}^d \times \tte; \mathbb{R}^d)$ be  a function such that 
\begin{equation*} 
|f(x;\te)| \le c \lp  1+|x|^N\rp,
\quad
|\partial_x f(x;\te)| \le c \lp  1+|x|^N\rp,  \quad   |\partial_{\vartheta} f(x;\te)| \le c \lp  1+|x|^N \rp
\end{equation*}
for some $c>0, N \in \mathbb{N}$, independent of $\te \in \tte$. Assume moreover that there exists a function $U \in \mathcal{C}^{2,1}(\mathbb{R}^d \times \tte; \mathbb{R})$ such that
$$ \partial_x U(x; \te) =f(x; \te), \qquad x \in \mathbb{R}^d, \,\, \te \in \tte,$$
i.e. $f$ is of gradient type.
Then, for $H>1/2$, we have
$$ \sup_{\vartheta \in \tte} \left|  \frac{1}{n \alpha_n} \sum_{k=0}^{n-1}   \langle  f(Y_{t_k};\te),  \delta F_{t_k t_{k+1} } \rangle  + \mathbf{E}\langle  b(\overline{Y}; \te_0), f(\overline{Y};\te) \rangle \right| \rightarrow 0  \qquad {\bf P}\textrm{-}a.s.$$
In particular,
\begin{equation*}
\sup_{\vartheta \in \tte} \left|  Q_n^{(2)}(\te) + \mathbf{E}\langle  b(\overline{Y}; \te_0), \delta_{\te_0\te}b(\overline{Y}) \rangle  \right| \rightarrow 0  \qquad {\bf P}\textrm{-}a.s.
\end{equation*}

\end{lemma}

\begin{proof} Let $T_n=n \alpha_n$.
First note that the chain of variable and density formula for Riemann-Stieltjes integrals, see (\ref{ito_proto}) and(\ref{dens_form}) in Subsection \ref{RS_integrals}, gives
that
\begin{equation*}
\frac{1}{T_n} \left( U(Y_{T_n}; \te) - U(y_0; \te) \right) = \frac{1}{T_n} \int_0^{T_n}  \langle  f(Y_u; \te),  b(Y_u; \te_0) \rangle \, du
    +   \frac{1}{T_n} \int_0^{T_n}   \langle  f(Y_u; \te),  d F_u\rangle.
\end{equation*}
Now the properties of $f$, Proposition \ref{prop:bnd-moments-Y} and Lemma \ref{pathwise} imply that
$$ \sup_{\vartheta \in \tte}  \frac{1}{T_n} \left|  U(Y_{T_n};\te) -  U(y_0;\te) \right| \rightarrow 0  \qquad {\bf P}\textrm{-}a.s. $$ 
Moreover, we have 
$$ \sup_{\vartheta \in \tte}   \left| \frac{1}{T_n} \int_0^{T_n}  \langle f(Y_u; \te),  b(Y_u; \te_0) \rangle \, du - \mathbf{E} \langle  f( \overline{Y};\te ),  b( \overline{Y}; \te_0) \rangle \right| \rightarrow 0  \qquad {\bf P}\textrm{-}a.s.,$$
which  can be derived completely analogously  to (\ref{uniform_ergod}).
It follows
\begin{align*}
\sup_{\vartheta \in \tte}   \left| \frac{1}{T_n} \int_0^{T_n}   \langle f(Y_u;\te),  d F_u \rangle 
 + \mathbf{E} \langle f( \overline{Y};\te),  b( \overline{Y}; \te_0) \rangle \right| \rightarrow 0  \qquad {\bf P}\textrm{-}a.s. 
\end{align*}
So, it remains to show that
\begin{align}\label{to_show_F}
\sup_{\vartheta \in \tte}   \frac{1}{T_n} \left| \int_0^{T_n} \langle G_n(t; \te), dF_t \rangle   \right| \rightarrow 0  \qquad {\bf P}\textrm{-}a.s. 
\end{align}
where
$$ G_n(t;\vartheta)=f(Y_{t};\te) - f(Y_{t_k};\te) , \qquad t \in [t_k,t_{k+1}), \qquad k=0,1, \ldots.
$$
Applying Proposition \ref{young_est} and using the polynomial Lipschitz continuity of $f$ yields, for all $\la<H$,
\begin{align*} \left|   \int_0^{T_n} \langle G_n(t; \te), dF_t \rangle   \right| \leq  c \cdot \alpha_n^{2\lambda} \cdot \sum_{j=1}^{m}  \sum_{k=0}^{n-1} \sup_{t \in [t_k, t_{k+1}]} (1 + |Y_t|^{N}) | Y  |_{\lambda; [t_k;t_{k+1}]}   | B^{(j)}  |_{\lambda; [t_k;t_{k+1}]}  . \end{align*}
 From the Garcia-Rademich-Rumsey inequality, see Lemma \ref{grr_lemma}, and  Proposition \ref{prop:bnd-moments-Y}  we have that
$$ \left( \mathbf{E} | Y  |_{\lambda; [t_k;t_{k+1}]}^p \right)^{1/p}  \leq c \cdot \alpha_n^{H-  \lambda}  $$
and 
also
$$ \left( \mathbf{E} | B^{(j)}  |_{\lambda; [t_k;t_{k+1}]}^p \right)^{1/p}  \leq  c \cdot \alpha_n^{H- \lambda}. $$
 Since moreover
$$ \sup_{t \in [t_k, t_{k+1}]} (1 + |Y_t|^{N})  \leq c \cdot \left( Y_{t_k}^N + \alpha_n^{\lambda N} \cdot |Y|_{\lambda; [t_k,t_{k+1}]}^N   \right)$$
and 
$ \sup_{t \geq 0} \mathbf{E} |Y_t|^p < \infty $
for all $p \geq 1$, it follows that
\begin{align}\label{eq:cvgce-Gn-dF}  \left( \mathbf{E} \sup_{ \te \in \tte} \left| \frac{1}{T_n}  \int_0^{T_n} \langle G_n(t; \te), dF_t \rangle   \right|^p \right)^{1/p} \leq  c \cdot   \alpha_n^{2H-1} \end{align}
Now  Lemma \ref{pathwise} implies \eqref{to_show_F}, since $H>1/2$.

\end{proof}

\begin{remark}\label{rmk:H-greater-one-half}
As mentioned in the introduction, the condition $H>1/2$ is used in  our proofs. Specifically, it is invoked in the convergence of the weighted stochastic integral~\eqref{eq:cvgce-Gn-dF} and also in order to derive \eqref{eq:rk3}.
\end{remark}

\smallskip

The following Lemma deals with the remaining term, i.e. the quadratic variations of the process $F$:

\begin{lemma}\label{inc_fbm} We have
$$ \lim_{n \rightarrow \infty} Q_n^{(3)}
=\lim_{n \rightarrow \infty} \frac{1}{n \alpha_n^2} \sum_{k=0}^{n-1} \left( |\delta F_{t_k t_{k+1}}|^2- \| \sigma \|^2 \alpha_n^{2H} \right) = 0 \qquad \mathbf{P}\textrm{-}a.s.$$
with
$\|\sigma \|^2 = \sum_{j=1}^m |\sigma_j|^2$.
\end{lemma}

\begin{proof} We have
\begin{align*}
|\delta F_{t_k t_{k+1}}|^2- \| \sigma \|^2 \alpha_n^{2H} =  \underbrace{\sum_{j=1}^m  |\sigma_j|^2 \left( |\delta B_{t_kt_{k+1}}^{(j)}|^2 - \alpha_n^{2H} \right)}_{=I_k^{(1)}} +  \underbrace{\sum_{i,j=1, \, i \neq j}^m  \langle  \sigma_i ,  \sigma_j \rangle \delta B_{t_kt_{k+1}}^{(i)} \delta B_{t_kt_{k+1}}^{(j)}}_{=I_k^{(2)}} .
\end{align*}
Owing to the scaling property of fBm it follows that
\begin{align*}
\mathbf{E} \left| \sum_{k=0}^{n-1} I_k^{(1)}  \right|^p
= \alpha_n^{2Hp}
\mathbf{E} \left|  \sum_{k=0}^{n-1} \sum_{j=1}^m |\sigma_j|^2 \big[ |\delta B_{k k+1}^{(j)}|^2- 1  \big]  \right|^p.
\end{align*}
Since all moments of random variables in a finite Gaussian chaos are equivalent, it follows from
(\ref{BM1})-(\ref{BM3})  that
 $$   \mathbf{E} \left|  \sum_{k=0}^{n-1}   \big[|\delta B_{k k+1}^{(j)}|^2-1\big]   \right|^p \leq  c \cdot  |\log(n)|^p \cdot n^{p(2H-1)}$$ and consequently 
$$     \mathbf{E} \left|  \sum_{k=0}^{n-1} I_k^{(1)}  \right|^p \leq c \cdot \alpha_n^{2Hp} \cdot    |\log(n)|^p \cdot  n^{p(2H-1) }.
$$ Since $\alpha_n = \kappa \cdot n^{-\alpha}$ with $\alpha \in (0,1)$ Lemma \ref{pathwise} now implies that
$$ \lim_{n \rightarrow \infty} \frac{1}{n \alpha_n^2} \left| \sum_{k=0}^{n-1} I_k^{(1)} \right| = 0 \qquad \mathbf{P}\textrm{-}a.s.$$
So it remains to consider the off-diagonal terms, i.e. $I_k^{(2)}$. Here we can exploit the following trick:
Let $\beta$ and $\tilde{\beta}$ be two independent fractional Brownian motions with the same Hurst index. From (\ref{BM1})-(\ref{BM3})  we clearly have that
$$V_{n}=
\sum_{k=0}^{n-1}  \left(  | \delta_{t_k t_{k+1}}\beta|^2  -  | \delta_{t_k t_{k+1}} \tilde{\beta} |^2
 \right) $$
satisfies
$$ \mathbf{E} |V_{n}|^p \leq c \cdot \alpha_n^{2Hp} \cdot    |\log(n)|^p \cdot  n^{p(2H-1) }.$$ However,
setting $B^{(i)}=(\beta+\widetilde{\beta})/\sqrt{2}$ and $B^{(j)}=(\beta-\widetilde{\beta})/\sqrt{2}$, then
 $B^{(i)}$ and $B^{(j)}$ are two independent fractional Brownian motions and 
$$ V_n  \stackrel{\mathcal{L}}{=} 2 \sum_{k=0}^{ n -1} \delta_{t_k t_{k+1}}B^{(i)} 
\delta_{t_k t_{k+1}} B^{(j)}. $$
Now we can easily conclude that
$$ \lim_{n \rightarrow \infty} \frac{1}{n \alpha_n^2} \left| \sum_{k=0}^{n-1} I_k^{(2)} \right| = 0 \qquad \mathbf{P}\textrm{-}a.s.$$
\end{proof}

\begin{proof}[Proof of Theorem \ref{thm:lse-convergence_prev}]
Let us go back to expression \eqref{eq:exp-Q-n1} and \eqref{eq:exp-Q-n2}.
Applying Lemma \ref{ergodic:disc} and \ref{ergodic_sum:fbm} we obtain that
 \begin{align*}   \lim_{n \rightarrow \infty}  \sup_{\te \in \tte} \Big|
 \lp Q_n^{(1)}(\te)-2 Q_n^{(2)}(\te) \rp
 -   \left( \mathbf{E}|b(\overline{Y};\te)|^2 -  \mathbf{E}|b(\overline{Y};\te_0)|^2 \right)\Big|=0
\end{align*}
almost surely. Furthermore, recall that Lemma \ref{inc_fbm} asserts that $\lim_{n \rightarrow \infty} Q_n^{(3)}=0$ almost surely. The proof is now finished.

\end{proof}

\medskip

\begin{proof}[Proof of Theorem \ref{thm:lse-convergence}]
Let us first recall the following result (\cite{Fry, Kas}):
\begin{proposition}\label{prop:cvgce-Ln} 
Assume that the family of random variables $L_n(\vartheta)$, $n \in \mathbb{N}$, $\vartheta \in \tte$, satisfies:
\begin{itemize}
\item[(1)] With probability one, $L_n(\vartheta) \rightarrow L(\vartheta)$ uniformly in $\vartheta \in \Theta$ as $n \rightarrow \infty.$
\item[(2)] The limit $L$ is non-random and
$ L(\vartheta_0) \leq L(\vartheta)$   for all $ \vartheta \in \tte .$
\item[(3)] It holds $L(\vartheta)= L(\vartheta_0)$ if and only if $ \vartheta = \vartheta_0$.
\end{itemize}
Then, we have 
$$ \widehat{\vartheta}_n \rightarrow \vartheta_0 \qquad {\bf P}\textrm{-}a.s. $$
for $n \rightarrow \infty$, where
$$ L_n(\widehat{\vartheta}_n) = \min_{\vartheta \in \tte} L_n(\vartheta). $$
\end{proposition}

The strong consistency of the zero squares  estimator follows now from  Theorem~\ref{thm:lse-convergence_prev} and an application of Proposition  \ref{prop:cvgce-Ln} to  $|Q_n(\te)|$.

\end{proof}

\medskip

\begin{remark}\label{H12}
In the case $H=1/2$ we have under similar assumptions that
$$  \sup_{ \vartheta \in \Theta}  \left| \big( Q_n(\vartheta) - Q_n(\vartheta_0) \big)   - \left( \mathbf{E} \, |  b( \overline{Y}_0 ;\te_0)-  b( \overline{Y}_0 ;\te) |^{2}    \right)   \right| \rightarrow 0 $$ in the $\bp$-almost sure sense, see e.g. \cite{flor}. Since
\begin{align*}
Q_n(\te_0) & =    \frac{1}{n \alpha_n^2}\sum_{k=0}^{n-1} \left(
  |\der F_{t_{k} t_{k+1}}|^2  - \|\sigma \|^2 \alpha_n  \right )   +  \frac{1}{n \alpha_n^2}\sum_{k=0}^{n-1} \left|r_k \right|^2  + \frac{2}{n \alpha_n^2}\sum_{k=0}^{n-1} \langle  \delta F_{t_kt_{k+1}}, r_k \rangle, 
\end{align*}
an application of Lemma \ref{ergodic:disc} and \ref{inc_fbm} (which are also valid for $H=1/2$) yield that
$$  \lim_{n \rightarrow \infty} Q_n(\te_0) = \lim_{n \rightarrow \infty} \frac{2}{n \alpha_n^2}\sum_{k=0}^{n-1} \langle  \delta F_{t_kt_{k+1}}, r_k \rangle  \qquad {\mathbf{P}}\textrm{-}a.s. $$
However, using the It\^o-isometry and the Burkholder-Davis-Gundy inequality we have
$$ \mathbf{E} \left| \frac{1}{n \alpha_n^2} \sum_{k=0}^{n-1} \langle r_k , \delta_{t_kt_{k+1}} F \rangle \right|^p \leq c \cdot \frac{1}{n^p \alpha_n^{2p}} \cdot T_n^{1+p/2} \alpha_n^{3p/2}\leq c \cdot n^{-p/2 +1-\alpha} $$
and Lemma \ref{pathwise} thus gives
$$  \lim_{n \rightarrow \infty} Q_n(\te_0) =0 \qquad {\mathbf{P}}\textrm{-}a.s.$$
Hence we end up with
$$  \sup_{ \vartheta \in \Theta}  \left| Q_n(\vartheta)    - \left( \mathbf{E} \, |  b( \overline{Y}_0 ;\te_0)-  b( \overline{Y}_0 ;\te) |^{2}    \right)   \right| \rightarrow 0, $$
so the limit of the statistics $Q_n$ is different for $H=1/2$, where one obtains the standard least square estimator.
\end{remark}

\end{document}